\documentclass{amsart}
\usepackage{amssymb, amsmath, amscd, color}
\usepackage[utf8]{inputenc}
\usepackage{palatino}

\newcounter{parbf_num}
\newcommand{\parbf}[1]{ \paragraph{\textbf{\arabic{parbf_num}. #1}} \stepcounter{parbf_num} }

\newcounter{g_num}
\newtheorem{theorem}{Theorem}

\newtheorem{lemma}[theorem]{Lemma}

\newtheorem{remark}[theorem]{Remark}
\newtheorem{example}[theorem]{Example}

\begin{document}	
\title{Aperiodic tilings of manifolds of intermediate growth}

\author{Michał Marcinkowski \textbf{\textit{\&}} Piotr W. Nowak}

\address{Uniwersytet Wroc\l awski \textit{\&} Instytut Matematyczny Polskiej Akademii Nauk}
\email{marcinkow@math.uni.wroc.pl}

\address{Uniwersytet Warszawski \textit{\&} Instytut Matematyczny Polskiej Akademii Nauk}
\email{pnowak@mimuw.edu.pl}
\thanks{The second author was partially supported by a grant from the Foundation for Polish Science.}

\begin{abstract}
We give a homological construction of aperiodic tiles for certain open Riemannian surfaces admitting actions of Grigorchuk groups of intermediate growth.
\end{abstract}

\maketitle

 Let $X$ be a noncompact Riemannian manifold.  A set of tiles for $X$ is a triple $\{ \mathcal{T},\mathcal{W},o \}$,
where $\mathcal{T}$ is a finite collection of compact polygons with boundary (tiles), each with distinguished connected faces, 
$\mathcal{W}$ is a collection of all faces of $\mathcal{T}$ and $o \colon \mathcal{W} \to \mathcal{W}$ is an opposition (matching) function. 
A tiling of $X$ is a cover $X = \cup_{\alpha} X_{\alpha}$, where each $X_{\alpha}$ is isometric to a tile in $\mathcal{W}$, 
every non-empty intersection of two distinct pieces is identified with faces $w_\alpha$ and $w_\beta$ 
of the corresponding tiles and satisfies $o(w_\alpha) = w_\beta$.
A tiling is (weakly) aperiodic if no group acting on $X$ cocompactly by isometries preserves the tiling. 
An aperiodic set of tiles of $X$ is a set of tiles, which admits only aperiodic tilings.
Classical examples include aperiodic tiles of the Euclidean spaces, such as Penrose tiles of the plane.

Let $\widetilde{M} \to M$ be an infinite covering of a closed manifold.
In \cite{wb} Block and Weinberger constructed aperiodic tiles for $\widetilde{M}$ when the covering group is non-amenable. 
Their construction relies on the fact that, for such a group, its
uniformly finite homology with coefficients in $\mathbb{Z}$ is trivial in degree 0. Unfortunately, this homology group is highly nontrivial for amenable groups.
Apart from the Euclidean spaces, there are virtually no constructions of aperiodic tiles for amenable manifolds.
Recently the second author with S. Weinberger constructed aperiodic tiles for the real Heisenberg group by a different method (unpublished).

In this note we return to the original method of \cite{wb} and the vanishing of uniformly finite homology in degree 0. 
We construct, using homology with torsion
coefficients, aperiodic tiles for manifolds equipped with proper cocompact
actions of the Grigorchuk group or certain other groups of intermediate growth
It is
well known that such a manifolds are amenable (i.e., regularly
exhaustible).

\subsection*{Acknowledgments} We are grateful to Slava Grigorchuk for helpful corespondence, bringing  \cite{g2} to our attention and helpful comments on the first draft.
The second author would like to thank Shmuel Weinberger for many inspiring discussions.\\

\parbf{Coarse homology.} Coarse locally finite homology was introduced in \cite{roe}.
We briefly define the 0-dimensional coarse locally finite homology group $HX_0$, according to our context. 
Let $G$ be a finitely generated group, let $S$ be its symmetric generating set and let $\Gamma_G = (V_G, E_G)$ be the corresponding oriented Cayley graph. 
Given an abelian group $A$, let $CX_0(G,A)$ be the space of  functions $G\to A$ and 
define $CX_1(G,A)$ to be the space of 1-dimensional chains $c=\sum_{x,y\in G}c_{[x,y]}[x,y]$, such that 
for every $c$ there exists $R>0$ such that $c_{[x,y]}=0$ if $d(x,y)\ge R$.
Let $\partial:CX_1(G,A)\to CX_0(G,A)$ denote the usual differential and 
define $HX_0(G,A) = C_0(G,A)/\operatorname{Image}\partial$.
We will mainly be interested in the case $A=\mathbb{Z}_p$, the cyclic group of order $p$.
In this case the coarse locally finite homology and the uniformly finite homology of \cite{wb} agree.

\begin{lemma}\label{lemma: 0-hom vanishes}
Let $G$ be an infinite finitely generated group and let $p\in \mathbb{N}$. Then $HX_0(G,\mathbb{Z}_p) = 0$.
\end{lemma}
\begin{proof}
Let $T$ be a maximal tree in $\Gamma_G$. Fix a root and orient the edges away from the root.  
For a 0-cycle $c$, we construct a 1-chain $\psi$, supported on the edges of the tree, satisfying $\partial \psi = c$.
If there are only finitely many vertices under the edge $e$, define $\psi(e)$ to be the sum of the values of $c$ on the vertices laying beneath $e$. 
We see that $c-\partial \psi$ is zero on the final vertices of the edges as above. Remove these vertices from $T$.
Now consider an infinite ray $\gamma$ starting from the root. $F = T \backslash \gamma$ is a forest of infinite trees 
(the finite components were truncated it in the first step).
It is obvious, that any 0-cycle with coefficients in $\mathbb{Z}_p$ supported on a ray is a boundary 
(we can solve the equation $\partial \psi = c$ consecutively starting from the initial vertex of the ray).
Modify $c$ so that it is non-zero only on $F$. 
We continue as above on each component of $F$,
constructing $\psi$ which bounds $c$.
\end{proof}

\begin{remark}\normalfont
A slightly different, but an ultimately longer proof, can be given using the homological Burnside theorem, a positive solution to a weaker, homological version of the Burnside problem 
on existence of torsion groups \cite[Theorem 3.1]{ns}.
More precisely, on every infinite finitely generated group the fundamental class vanishes in linearly controlled homology with integral coefficients.
\end{remark}
 
\parbf{Construction of the tiles} Let $(\widetilde{M},\widetilde{d}) \to (M,d)$ be an infinite cover of a Riemannian manifold $M$ and let $G$ be the covering group. 
Consider a Dirichlet domain for the action  of $G$ (for some arbitrary $x_0 \in \widetilde{M}$):
\begin{equation*}
D = \{ x \in \widetilde{M} \colon \widetilde{d}(x,x_0) \leq \widetilde{d}(x,g.x_0) \textrm{ for all } g \in G \},
\end{equation*}

\noindent together with a collection of  faces, $W_g = D \cap g.D$ (analogously, we define  faces for every translation $g.D$). 
By the Poincar\'{e} lemma, the finite set $S = \{ g : W_g \neq \emptyset \}$ generates $G$. 
The graph whose vertices correspond to the translations of $D$ by elements of $G$, 
with edges connecting $g.D$ and $h.D$ if and only if $g.D \cap h.D \neq \emptyset$,
is isomorphic to a Cayley graph for $S$. (It is convenient to think that a vertex $g$ of the Cayley graph lays inside $g.D$ and edges labelled by generators pass 
though faces.)

Lemma \ref{lemma: 0-hom vanishes} provides $\psi$, which satisfies $\partial \psi = \sum_{g \in G} g \in HX_0(G,\mathbb{Z}_p)$. For each oriented edge $e$
we decorate a face crossed by $e$ by adding $\psi(e)$ bumps along $e$ (thus a face of a tile where $e$ ends has $\psi(e)$ matching indentations). 
The sets $S$ and $\mathbb{Z}_p$ are finite and performing the above modifications gives only finitely many different tiles.\\

\parbf{The main theorem}
We will consider Dirichlet domains $D$ as above that satisfy a \emph{grid condition}, that  the tiling of $\widetilde{M}$ by translates of $D$ is unique. Such a condition is easy to enforce in various settings, by e.g., considering manifolds $M$ with a trivial isometry group, requiring that the tiles respect 
some triangulation of $\widetilde{M}$ or taking a subgroup of sufficiently large finite index in the non-simply connected case. Many known examples of tilings of Euclidean 
spaces, as well as those constructed in \cite{wb}, are grid tilings in the above sense.

\begin{theorem}\label{theorem: index not divisible}
Let $G$ be an infinite, finitely generated group. Assume that there exists 
$p \in \mathbb{N}$, such that  for every finite index subgroup $H$ of $G$, $p$ is not a factor of  $[G : H]$.
Let $M$ be a compact manifold  and $\widetilde{M}$ be a regular covering of $M$, on which $G$  acts by deck transformation and the fundamental domain satisfies the grid condition.
Then $\widetilde{M}$  admits an aperiodic set of tiles. 
\end{theorem}

\begin{proof} 

Apply the above construction of a finite set of tiles to $\widetilde{M} \to M$ with some $\psi$, satisfying
$\partial \psi =  \sum_{g \in G} g \in HX_0(G,\mathbb{Z}_p)$. Now choose any tiling of $\widetilde{M}$. From the grid condition we can assume that 
every tile (modulo modifications) is a translation of $D$. 
Define the chain $\psi'$ as follows: $\psi'(e)=$ the number of
bumbs on the face crossed
by $e$, with the appropriate sign.
By the definition of the matching rules we have $\partial \psi' =  \sum_{g \in G} g$. 
(Note that in general, $\psi'$ might be different from $\psi$ if the tiling is different from the one which appears in the construction.)

Assume now that $G'$ acts cocompactly by isometries, respecting the tiling. 
$G'$ also acts on $\Gamma$ and, therefore, is a finite index subgroup of $G$. 
Observe that $\psi'$ descends to $\Gamma /G'$ and we have the following equality in $\mathbb{Z}_p$:
\begin{equation*}
[G:G'] = \sum_{v \in \Gamma /G'} 1 = \sum_{e \in \Gamma /G'} \psi'(e) + \psi'(e^-) = 0.
\end{equation*}
Since  $p$ does not divide  $[G:G']$ we get a contradiction.
\end{proof}

We will now construct interesting examples of amenable manifolds to which our theorem applies. 
Let $G$ be a finitely  generated group. There exists a closed manifold $M$ and a regular covering $\widetilde{M} \to M$ with $G$ acting by deck transformations. 
Indeed, given a projection $p : H \to G$,
where $H$ is finitely presented (e.g., free group) we can take a compact manifold $M$ with  the fundamental group $\pi_1(M)=H$ 
and a regular covering corresponding to $ker(p)\subseteq H$.
Note that $M$ can be chosen to be a closed oriented surface with sufficiently large genus $2g$
(there is a projection of the surface group onto the free group $F_g$) and that by smooth modifications of the metric on $M$ we 
can ensure the grid condition for a Dirichlet domain in $\widetilde{M}$.

\begin{example}\normalfont
The Grigorchuk groups of intermediate growth are amenable, residually finite, finitely generated torsion 2-groups.
For the definition see \cite{g1} and for the careful construction of surfaces on which those groups act by covering actions see \cite[\S 4.]{g2}. 
By applying the above construction with $p=3$ we obtain, by Theorem \ref{theorem: index not divisible}, aperiodic tiles for such coverings. 
\end{example}

\begin{example}\normalfont
In \cite{fg} a finitely generated, residually finite group of intermediate growth was constructed. This group, in contrast to the Grigorchuk groups,
is virtually torsion-free \cite[Theorem 6.4]{bg} and every finite quotient is a 3-group \cite[Theorem 6.5]{bg}. By a similar construction as before we obtain surfaces
on which such groups act properly cocompactly and Theorem \ref{theorem: index not divisible}  with $p=2$ provides aperiodic tiles for such manifolds.
\end{example}

\end{document}